\documentclass[12pt,a4paper]{amsart}
\usepackage{amsfonts}
\usepackage{amsmath}
\usepackage{amssymb}
\usepackage{amsthm}
\usepackage{graphicx}
\usepackage{hyperref}
\usepackage[T1]{fontenc}

\newtheorem{theorem}{Theorem}

\newtheorem{observation}[theorem]{Observation}
\newtheorem{corollary}[theorem]{Corollary}
\newtheorem{problem}[theorem]{Problem}

\begin{document}
	\title{On-line majority edge-colourings of graphs}
	\thanks{Research partially supported by the AGH University of Krakow under grant no. 16.16.420.054 and Excellence Initiative -- Research University AGH University of Krakow, funded by the Polish Ministry of Science and Higher Education.}
	
	\author{Pawe{\l} P\k{e}ka{\l}a}
	\address[agh]{AGH University of Krakow, al. A. Mickiewicza 30, 30-059 Krakow, Poland}
	\email{\tt ppekala@agh.edu.pl}
	
	\begin{abstract}
		A \emph{majority edge-colouring} of a graph $G$ is a colouring of the edges of $G$ such that, for every vertex $v$ of $G$, at most half of the edges incident with $v$ receive the same colour. This notion was introduced by Bock \textit{et al.} in 2023, who proved that every graph of minimum degree at least $2$ has a majority $4$-edge-colouring.
		
		We investigate an \emph{on-line} variant of majority edge-colouring in which the graph is revealed by the Presenter edge-by-edge and the Algorithm must colour each edge immediately and irrevocably. In particular. we prove that the greedy strategy, which uses at most five colours, is optimal among on-line algorithms if $\delta = O(\frac{\log n}{\log\log n})$. We further extend our results to $1/k$-majority edge-colourings.
	\end{abstract}
	
	\maketitle
	
	\section{Introduction}
	
	Majority colourings are a relaxation of proper colourings. Whereas a proper colouring requires that all neighbours (vertices or edges) of a given vertex receive distinct colours, a majority colouring only requires that no colour is present on majority of those neighbours. In recent years many variants of this problem received considerable attention.
	
	The notion of majority colourings was first considered in the context of digraphs by Kreutzer, Oum, Seymour, van der Zypen and Wood~\cite{digraph1}. In this setting, a majority colouring is a colouring of the vertices of a digraph $D$ such that for every vertex $v$ of $D$ at most half of the out-neighbours of $v$ receive the same colour as $v$. They proved that four colours always suffice and conjectured that three colours are enough. See also \cite{Szabo-majority,digraph2,MajorityGeneralOur,digraph3,Knox-Samal} for related results and generalisations.
	
	Note that the optimal result for the corresponding variant in the graph setting follows from the Lov\'asz's decomposition theorem~\cite{Lovasz}, which implies that every finite graph admits a majority colouring with only two colours. A stronger variant in this setting, which requires that \textit{no} colour constitutes a majority in the neighbourhood of any vertex, was recently introduced by Kalinowski, Kamyczura, Pil\'sniak and Wo\'zniak~\cite{strong_kkpw}.
	
	The extension of majority colourings to edges was introduced by Bock, Kalinowski, Pardey, Pilśniak, Rautenbach and Woźniak~\cite{Bock}. A \textit{majority edge-colouring} of a graph $G$ is an edge-colouring such that, for every vertex $v$ of $G$, at most half the edges incident with $v$ share the same colour. Note that such a colouring cannot exist if a graph contains a vertex of degree $1$. It was shown in \cite{Bock} that four colours are enough for any graph without vertices of degree $1$, and that the number of colours can be reduced to three if $\delta (G) \geq 4$.
	
	\begin{theorem}[\cite{Bock}]
		Every graph of minimum degree at least $4$ has a majority $3$-edge-colouring
	\end{theorem}
	
	A colouring algorithm is called \textit{off-line} if the input (graph) can be accessed before the output (colouring) is produced. In contrast, an \textit{on-line} algorithm is presented with the input sequentially and must make decisions based only on the subset of the input seen so far, assigning colours immediately and irrevocably. There are two natural models of on-line colouring: the input is given either vertex-by-vertex or edge-by-edge. In either case, the algorithm assigns a colour to the current vertex or edge based only on the already revealed part of a graph, and the chosen colour cannot be changed later. In the most general setting, i.e., under \textit{adversarial} arrivals, the adversary determines the graph and the order in which it is revealed based on previous choices made by the algorithm. Other settings that have received significant attention include \textit{random-order} arrivals, where the adversary must decide on the graph in advance and the order is chosen uniformly at random from all possible permutations.
	
	The study of on-line colourings originated with vertex colouring. Early research focused on developing algorithms that could outperform the greedy algorithm, which assigns each new vertex $v$ the lowest-labelled colour not already assigned to a vertex adjacent to $v$. See e.g.~\cite{Gyarfas,Lovasz_online,Vishw} for important early results. The proper edge-colouring version was introduced by Bar-Noy, Motwani and Naor \cite{online}. They showed that the greedy edge-colouring algorithm, which uses at most $2\Delta(G) -1$ colours, is best possible, i.e. no on-line algorithm can colour all graphs using $2\Delta(G) -2$ colours, in either arrival method and regardless of whether the setting is adversarial or randomized. However, this result only holds for bounded-degree graphs (of maximum degree $\Delta = O(\log n)$ for the adversarial case and $\Delta = O(\sqrt{\log n})$ for the randomized case). Bar-Noy \textit{et al.} conjectured that better algorithms exist for graphs of sufficiently high maximum degree.
	
	The on-line edge-colouring conjecture has attracted considerable attention over the years. Recently, Blikstad, Svensson, Vintan and Wajc \cite{bsvw} resolved the conjecture, showing that a $(1+o(1))\Delta$-edge-colouring exists using a randomized algorithm for all graphs with $\Delta = \omega(\log n)$. Furthermore, they improved this bound in \cite{bsvw2} to $\Delta = \omega(\sqrt{\log n})$ in the randomized case and also solved the more general case by providing a deterministic on-line $(1+o(1))\Delta$-edge-colouring algorithm for graphs with known maximum degree $\Delta = \omega(\log n)$ in the adversarial edge arrival setting.
	
	While the off-line version of majority edge-colouring has been studied (\cite{Bock,majority_gen,majority_list}), its on-line analogue remains open. To the best of our knowledge, on-line majority colourings (whether vertex or edge colourings) have not been studied previously. We note, however, that Bosek, Grytczuk and Jakóbczak~\cite{Bosek} considered a related game version of majority vertex colouring in which two players colour vertices alternately, one player attempting to produce and the other to prevent a majority vertex-colouring of a graph using a prescribed number of colours.
	
	We formulate the on-line majority edge-colouring problem as a game between two players. The Presenter constructs a graph by revealing its edges one at a time. After each edge is revealed, the Algorithm must assign a colour to that edge immediately and irrevocably. The Presenter may choose the next edge based on the entire history of the game, including all colours assigned by the Algorithm.
	
	In the next section we establish that the greedy algorithm is optimal among the on-line algorithms if no additional constraints are introduced. Section \ref{sect_degree} extends this result to the case where the minimum degree of a graph is known, and Section \ref{sect_size} analyses the minimum number of vertices relative to that degree. The final section contains concluding remarks and open problems.

	\section{General case}
	
	In this subsection, we consider the adversarial version of on-line majority edge colouring in the \textit{edge-by-edge} model, where the only restriction on the Presenter is that the final graph must admit a majority edge colouring, i.e. the final graph has minimum degree at least $2$ (we assume for simplicity, that the final graph has no vertices of degree $0$). We start with an edgeless graph. Thus, during the game, the graph revealed so far may contain vertices of degree $1$ and therefore may not admit a majority edge colouring. Nevertheless, the Algorithm must colour the graph so that, in every round, the majority condition is satisfied at every vertex of degree at least $2$. Indeed, suppose that in some round the Algorithm assigns colour $c$ to the newly revealed edge $e$, and that after this round one endpoint $v$ of $e$ has degree at least $2$ and more than half of the edges incident with $v$ have colour $c$. If the graph after this round already has minimum degree at least $2$, the Presenter can end the game. Otherwise, in subsequent rounds the Presenter reveals edges incident with vertices of degree $1$ that are not incident with $v$. Hence, the vertex $v$ will still violate the majority condition in the final graph, and the Algorithm loses.
	
	Let $G$ be the graph formed by the edges revealed so far. We say that a colour $c$ is \textit{blocked} for a vertex $v$ of $G$ if $d_G(v) > 0$ and assigning colour $c$ to a new edge incident with $v$ would result in more than half of the edges incident with $v$ sharing the same colour. If $d_G(v) = 0$, no colour is blocked for $v$.
	
	We begin by showing that the \textit{greedy algorithm} allows the Algorithm to use at most $5$ colours. We recall that the greedy algorithm uses colours from the set $\{1,2,3,\dotsc\}$ and, for each edge revealed by the Presenter, assigns it the smallest colour that is not blocked at either endpoint.
	
	\begin{theorem}
		The Algorithm has a strategy using at most five colours against any strategy of the Presenter,
		\begin{proof}
			Consider a game in which the Algorithm uses the greedy algorithm. Suppose that, before the Presenter reveals the next edge, the currently revealed graph $G$ is coloured with at most five colours and that the majority condition is satisfied at every vertex of degree at least $2$.
			
			Let $e=uv$ be the next edge revealed by the Presenter. If $d_G(v)\geq 2$ is odd, then the threshold for the majority condition will increase after adding the edge $e$ and thus no colour is blocked at $v$; if $d_G(v)=1$ then exactly one colour is blocked. If $d_G(v)$ is even, then a colour $c$ is blocked at $v$ if and only if exactly half of the edges incident with $v$ are currently coloured with $c$. Hence, at most $2$ colours can be blocked at $v$, and similarly, at most $2$ colours are blocked at the vertex $u$. Thus, at most $4$ colours are blocked for the edge $e$, leaving at least one remaining colour available for the edge $e$.
		\end{proof}
	\end{theorem}
	
	From now on, we consider games in which the Presenter has an additional restriction: before the game begins, the Algorithm knows the minimum degree $\delta$ of the final graph. It is enough to consider even values of $\delta$. Indeed, the maximum allowed number of incident edges receiving the same colour is the same for vertices of degrees $2l$ and $2l+1$, for every positive integer $l$. Hence, if the Presenter has a strategy forcing a given number of colours when $\delta=2l$, the same strategy can be extended to the case $\delta=2l+1$ by adding one edge incident with each vertex of degree $2l$ to the final graph.
	
	The smallest non-trivial minimum degree to consider is $\delta=2$, since we assume that no vertices of degree $0$ are present in the final graph (and those vertices are irrelevant for majority edge-colourings anyway, as they always satisfy the majority condition) and vertices of degree $1$ preclude the existence of any majority colouring, regardless of the number of colours.
	
	For subcubic graphs, majority edge-colouring is equivalent to proper edge-colouring, since vertices of degree $2$ and $3$ can have at most one incident edge in each colour. Thus, we can use the construction by Bar-Noy \textit{et al.} \cite{online} to prove the following observation.
	
	\begin{observation}\label{online2}
		The Presenter has a strategy forcing the Algorithm to use at least five colours if the final graph has minimum degree $\delta=2$.
		\begin{proof}
			Suppose that the Algorithm has a strategy using only four colours, regardless of the Presenter's strategy, provided that the final graph has minimum degree $2$.
			
			The Presenter begins by revealing the edges of 
			\[ 2 \binom{4}{2}+1=13 \]
			disjoint stars of degree $2$. 
			
			Because the centre of every star has degree $2$ and must satisfy the majority condition, the two edges incident with each centre receive distinct colours. Thus, each star determines an unordered pair of colours. Since there are only $\binom{4}{2}$ possible pairs, the Pigeonhole Principle implies that at least three stars receive the same pair of colours, $\{c_1,c_2\}$.
			
			Let $v_1,v_2,v_3$ be the centres of these three stars. Both $c_1$ and $c_2$ are blocked at each $v_i$. The Presenter adds a new vertex $v$ and reveals the three edges joining it to $v_1,v_2,v_3$. For the majority condition to hold at each $v_i$, each of the three new edges must receive a colour different from both $c_1$ and $c_2$. Moreover, since $v$ has degree $3$, these three edges must receive pairwise distinct colours. Since the Algorithm has only four colours available and only two of them are different from $c_1$ and $c_2$, the Algorithm cannot colour all three new edges. Thus, the Presenter forces the Algorithm to use a fifth colour, a contradiction.
		\end{proof}
	\end{observation}
	
	Note that this method does not extend easily to larger values of $\delta$. For example, if $\delta=4$, the Presenter can reveal many stars of degree $4$, but the Algorithm may colour the edges of each star with four distinct colours, which means no colour would be blocked at the centres. To obtain a contradiction, the Presenter would need to force the Algorithm to create at least five vertices of degree $4$ at which two colours are blocked, for example by ensuring that two incident edges receive colour $c_1$ and the other two receive colour $c_2$. In the next section we provide a construction achieving this goal.

	\section{Graphs with bounded minimum degree}\label{sect_degree}
	
	In this section, we extend Observation \ref{online2} to the case in which the Presenter must construct a final graph with minimum degree equal to a fixed integer $\delta \geq 2$. As established in the previous section, it suffices to consider even values of $\delta$. Thus, we write $\delta = 2l$ for some integer $l$.
	
	Similarly as before, vertices of degree smaller than $\delta$ may appear during the game. For such vertices, instead of requiring the majority condition to hold, we require that, for every vertex $v$ of degree smaller than $\delta=2l$, at most $l$ currently revealed incident edges have the same colour. If, during the game, the Algorithm colours an edge so that some vertex $v$ of degree smaller than $\delta$ has more than $l$ incident edges of the same colour, then the Presenter can reveal additional edges incident with $v$ until its degree becomes exactly $\delta$. The majority condition is then violated at $v$.
	
	If, on the other hand, at some point during the game the majority condition is not satisfied at a vertex $v$ of degree at least $\delta$, then the Presenter can reveal arbitrary edges incident with vertices of degree smaller than $\delta$ that are not incident with $v$. Hence, $v$ still violates the majority condition in the final graph, and the Algorithm loses.
	
	We now show that five colours are necessary for every fixed value of $\delta$.
	
	\begin{theorem}\label{main}
		Let $l\geq 1$ be an integer. Consider the on-line majority edge-colourings game of graphs with minimum degree $\delta=2l$. Then the Presenter has a strategy forcing the Algorithm to use at least five colours.
		\begin{proof}
			Suppose, for a contradiction, that the Algorithm has a strategy using only four colours, from the set $\{1,2,3,4\}$, regardless of the Presenter's strategy, provided that the final graph has minimum degree $\delta = 2l$.
			
			The Presenter's strategy is divided into several phases. The order in which edges are revealed within each phase is irrelevant. In the first phase, the Presenter reveals a star with
			\[ 2\cdot(2l+1)^{2l} \]
			leaves. Among the four colours, two colours occur most frequently on the edges of the star. Relabelling the colours if necessary, we may assume that these two colours are $1$ and $2$. Thus, at least $(2l+1)^{2l}$ of the leaves are incident with an edge coloured $1$ or $2$. Let $S_1$ be a set of exactly $(2l+1)^{2l}$ such leaves, chosen arbitrarily.
			
			In the second phase, the Presenter partitions the set $S_1$ into disjoint subsets of size $2l+1$. For each subset, the Presenter reveals a new vertex adjacent to all vertices in that subset. This creates 
			\[ (2l+1)^{2l-1} \]
			new vertices, each of degree $2l+1$. Hence, no colour can occur more than $l$ times in a majority edge-colouring. Since there are only two colours, $1$ and $2$, they can together occur on at most $2l$ incident edges. Hence at least one edge incident with each newly revealed vertex must receive a colour from $\{1,2\}$. From each $(2l+1)$-element subset of $S_1$, the Presenter selects one endpoint of such an edge and places it in a set $S_2$. Thus, after this step, the set $S_2$ contains $(2l+1)^{2l-1}$ vertices of degree two whose incident edges are coloured with $1$ or $2$.
			
			The Presenter now repeats this procedure. More precisely, suppose that, immediately before the phase $i$, the Presenter has a set $S_{i-1}$ containing 
			\[ (2l+1)^{2l-i+2} \]
			vertices, each of degree $i-1$, and every edge incident with a vertex in $S_{i-1}$ has colour $1$ or $2$. The Presenter partitions the set $S_{i-1}$ into disjoint subsets of size $(2l+1)$. For each such subset, the Presenter reveals a new vertex adjacent to all vertices in that subset. There are $(2l+1)^{2l-i+1}$ new vertices. Each has degree $(2l+1)$, so at least one of its incident edges must receive colour $1$ or $2$. The Presenter selects one endpoint of such an edge from each subset and places the selected vertices in the set $S_i$. Therefore, $S_i$ contains 
			\[ (2l+1)^{2l-i+1} \]
			vertices of degree $i$, all of which are incident only to edges coloured with $1$ or $2$.
			
			After the $2l$-th phase, the Presenter has constructed a set $S_{2l}$ of $2l+1$ vertices, each of degree $2l$, such that all edges incident with these vertices are coloured exclusively with colours $1$ and $2$. Since each such vertex already has degree $2l=\delta$ and the majority condition must hold at every vertex of degree at least $\delta$, both colours $1$ and $2$ occur exactly $l$ times among the edges incident with each vertex. In particular, both colours are blocked at every vertex in $S_{2l}$.
			
			Finally, the Presenter reveals a single new vertex $v$ adjacent to every vertex in $S_{2l}$. Since the vertex $v$ has degree $2l+1$, the edges incident with $v$ must receive at least three distinct colours in a majority edge-colouring. However, all its adjacent vertices have the colours $1$ and $2$ blocked, hence only colours $3$ and $4$ are available. This implies that a fifth colour is necessary, a contradiction.
		\end{proof}
	\end{theorem}

	\section{Relation between minimum degree and the size of the graph}\label{sect_size}
	
	In the previous section, we considered the on-line majority edge-colourings game with the minimum degree as the only parameter. We assumed that the number of available vertices was sufficiently large for the Presenter to construct a graph forcing the use of a fifth colour. Usually, in on-line colouring problems, the number of vertices of the final graph is also known in advance. In this section, we investigate the relation between the minimum degree and the number of vertices that allows the Presenter to force the Algorithm to use five colours. We begin with the following observation determining the number of vertices used in the construction from Theorem \ref{main}.
	
	\begin{observation}
		The graph constructed in Theorem \ref{main} requires 
		\[ \Theta(\delta^{\delta}) \]
		vertices.
		\begin{proof}
			Let $\delta = 2l$. The first phase uses $1+2\cdot(2l+1)^{2l}$ vertices: one centre of the star and $2\cdot(2l+1)^{2l}$ leaves. The construction has $2l+1$ phases in total, with the last phase adding a single vertex. Moreover, for each $i=2,\ldots,2l$, phase $i$ adds $2l+1$ times as many vertices as phase $i+1$. Thus, the total number of vertices is
			\begin{align*}
				 n&= 1+2\cdot(2l+1)^{2l} + \sum_{i=0}^{2l-1} (2l+1)^i\\
				&= 1+2\cdot(\delta+1)^{\delta} + \sum_{i=0}^{\delta-1} (\delta+1)^i\\
				&= 1+2\cdot(\delta+1)^{\delta} + \frac{(\delta+1)^{\delta}-1}{\delta}.
			\end{align*}
			As $\delta\to \infty$, the term $2(\delta+1)^{\delta}$ dominates the other two terms. Hence,
			\[n \sim 2(\delta+1)^{\delta}. \]
			
			Furthermore, since
			\[ 2(\delta+1)^{\delta} = 2\delta^{\delta}\left(1+\frac{1}{\delta}\right)^{\delta} \]
			and
			\[ \lim_{\delta\rightarrow \infty}\left(1+\frac{1}{\delta}\right)^{\delta}=e, \]
			we obtain $n \sim 2e\delta^{\delta}$.
			
			Thus, $n=\Theta (\delta^{\delta})$.
		\end{proof}
	\end{observation}
	
	It follows that if $\delta \geq 2$ and $n = \Omega(\delta^{\delta})$, the greedy algorithm is an optimal strategy for the Algorithm. Expressing $\delta$ in terms of $n$, we obtain the following.
	
	\begin{corollary}
		If $\delta$ is sufficiently small, namely 
		\[ \delta = O \left(\frac{\log n}{\log\log n} \right), \]
		then there is no strategy for the Algorithm using at most four colours against every strategy of the Presenter for $n$-vertex graphs with minimum degree $\delta$.
		\begin{proof}
			The construction in Theorem \ref{main} uses $n=\Theta(\delta^{\delta})$ vertices.
			Taking the natural logarithm on both sides we get
			\[ \log n = \delta \log \delta + O(1). \]
			The solution of $x\log x=\log y$ is $x=\frac{\log y}{W_0(\log y)}$, where $W_0$ denotes the principal branch of the Lambert $W$ function.
			Applying this identity gives
			\[ \delta \sim \frac{\log n}{W_0(\log n)}. \]
			Note that for $x\to \infty$ we have $W_0(x)=\log x - \log\log x + o(1)$. Thus for $n \to \infty$ we obtain
			\[ \delta \sim \frac{\log n}{\log\log n - \log\log\log n} \sim \frac{\log n}{\log\log n}.\]
			Since $n=\Omega(\delta^{\delta})$ gives a lower bound on $n$, we obtain an upper bound on $\delta$. Therefore, for
			\[ \delta = O\left(\frac{\log n}{\log\log n}\right) \]
			the Presenter has a strategy forcing the Algorithm to use at least five colours.
		\end{proof}
	\end{corollary}

	\section{Generalisation}
	Intuitively, results on off-line majority edge-colourings suggest that the number of colours required in the on-line setting should also depend on the minimum degree. However, we have shown that for arbitrarily large minimum degree, the Presenter can force the use of five colours. However, such a dependence on the minimum degree already arises for $k=3$ in the case of \textit{$\frac{1}{k}$-majority on-line edge-colouring}. This notion is a natural extension of the majority colourings considered previously in various settings, e.g. in \cite{MajorityGeneralOur,Bock,digraph3,Knox-Samal,digraph1}.
	
	For a fixed integer $k\geq2$, consider the game in which the Presenter reveals edges of a graph one by one so that the final graph has minimum degree $\delta\geq k$, where $\delta$ is fixed and known before the start of the game by the Algorithm. The goal of the Algorithm is to colour the edges with as few colours as possible so that the resulting colouring is a $\frac{1}{k}$-majority edge-colouring of the final graph. That is, for every vertex $v$, at most $d(v)/k$ incident edges may receive the same colour.
	
	As before, we must assume that the final graph has minimum degree at least $k$. Otherwise, if the final graph contained a vertex of positive degree smaller than $k$, no $\frac{1}{k}$-majority edge-colouring could exist, regardless of the number of colours (as before, vertices of degree $0$ are ignored).
	
	We begin by determining the number of colours sufficient for the Algorithm when using the greedy strategy. We say that a colour $c$ is \textit{blocked} for a vertex $v$ if adding an edge incident with $v$ and coloured $c$ would violate the $\frac{1}{k}$-majority condition at $v$.
	
	\begin{theorem}
		Let $k\geq 2$ be an integer. Consider the game of $\frac{1}{k}$-majority on-line edge-colouring of a graph in which the Presenter must construct a final graph with minimum degree $\delta=k\ell+j$, where $j\in \{0,\dotsc,k-1 \}$. Then the Algorithm has a strategy using at most
		\[ 2\left( k+\left\lfloor \frac{k-2}{\ell} \right\rfloor \right)+1 \]
		colours against any strategy of the Presenter.
		\begin{proof}
			Let $t = k+\left\lfloor\frac{k-2}{\ell}\right\rfloor$. Consider a game in which the Algorithm uses the greedy algorithm. Suppose that, before the Presenter reveals the next edge, the currently revealed graph $G$ is coloured with at most $2t+1$ colours and that the $\frac{1}{k}$-majority condition is satisfied at every vertex of degree at least $\delta$.
			
			Let $e=uv$ be the next edge revealed by the Presenter, and write $d_G(v)=kp+i$, where $p\geq \ell$ and $i\in \{0,\dotsc,k-1 \}$. If $d_G(v) \equiv k-1 \pmod{k}$, then the maximum allowed number of incident edges sharing a colour would increase after adding $e$, and hence no colour is blocked for $v$. Otherwise, a colour $c$ is blocked for $v$ if and only if $p$ incident edges at $v$ have already received colour $c$. Thus, at most
			\[ \left\lfloor\frac{d_G(v)}{p}\right\rfloor = k+\left\lfloor\frac{i}{p}\right\rfloor \leq k+\left\lfloor\frac{k-2}{p}\right\rfloor \leq k+\left\lfloor\frac{k-2}{\ell}\right\rfloor =t \]
			colours are blocked for $v$. If $d_G(v)<k\ell$, then at most $\lfloor \frac{d_G(v)}{\ell} \rfloor \leq k-1$ colours are blocked. The same argument applies to $u$, so at most $t$ colours are blocked there as well.
			
			Therefore, at most $2t$ colours are blocked at the endpoints of $e$. Among the $2t+1$ available colours, at least one is available at both endpoints. The Algorithm can assign such a colour to $e$. Thus, the $\frac{1}{k}$-majority condition is preserved after every move.
		\end{proof}
	\end{theorem}
	
	Thus, the larger $\delta$ is, the fewer colours the Algorithm needs when using this strategy. In particular, for $\delta=k=k\cdot1$ we have
	\[ 2\left(k+\left\lfloor\frac{k-2}{1}\right\rfloor\right)+1 = 2(2k-2)+1 = 4k-3, \]
	whereas for $\delta=k\ell\geq k(k-1)$ we obtain
	\[ 2\left(k+\left\lfloor\frac{k-2}{\ell}\right\rfloor\right)+1 = 2k+1. \]
	
	Note that for $k=2$,
	\[ 4k-3=2k+1=5.\]
	
	In the remainder of this section, we briefly discuss a construction, analogous to the one in Theorem \ref{main}, that forces exactly $2\left(k+\lfloor\frac{k-2}{\ell}\rfloor\right)+1$ colours. We omit the details, as the argument is analogous to that used in the proof of Theorem \ref{main}.
	
	If a vertex $v$ has degree $k\ell+i$, where $\ell$ is an integer and $i\in \{0,\dotsc,k-1 \}$, then a vertex of degree $k\ell+i$ permits the same number of incident edges of any one colour as a vertex of degree $k\ell$. Hence, it suffices to consider the case in which the minimum degree of the final graph is divisible by $k$. Indeed, if the Presenter has a strategy forcing a given number of colours when $\delta=k\ell$, then the same strategy can be extended to the case $\delta=k\ell+i$, where $0<i<k$, by adding edges to the final graph obtained for $\delta=k\ell$ until we increase its minimum degree to $k\ell+i$, analogously to the argument given above for $k=2$.
	
	Moreover, as before, for a fixed minimum degree $\delta=k\ell$, if in some round the Algorithm coloured an edge so that one of its endpoints had degree at least $\delta$ and the $\frac{1}{k}$-majority condition was violated there, or if a vertex had degree smaller than $\delta$ but more than $\ell$ incident edges of the same colour, then the Presenter would have a winning strategy: they could add edges in such a way that the colouring produced by the Algorithm would fail to be a $\frac{1}{k}$-majority edge-colouring of the final graph, regardless of the Algorithm's subsequent choices.
	
	\begin{theorem}
		Let $k\geq 2$ be an integer. Let $\delta=k\ell$ and $t=k+\left\lfloor\frac{k-2}{\ell}\right\rfloor$. Consider the game of $\frac{1}{k}$-majority on-line edge-colouring of a graph with minimum degree $\delta$. Then the Presenter has a strategy forcing the Algorithm to use at least $2t+1$ colours.
		\begin{proof}[Proof sketch]
			For a vertex of degree $t\ell+1$, exactly $\ell$ incident edges may have the same colour, and hence
			\[ \left\lceil\frac{t\ell+1}{\ell}\right\rceil=t+1 \]
			colours are necessary in a $\frac{1}{k}$-majority edge-colouring.
			
			Suppose that the Algorithm has a strategy using only $2t$ colours, from the set $\{1,2,\dotsc,2t\}$, regardless of the Presenter's strategy, provided that the final graph has minimum degree $\delta$.
			
			In the first phase, the Presenter presents a star with
			\[ 2\cdot(t\ell+1)^{t\ell} \]
			leaves. Without loss of generality, assume that at least $(t\ell+1)^{t\ell}$ edges of this star have been coloured with colours from $\{1,2,\dotsc,t\}$. The Presenter selects exactly $(t\ell+1)^{t\ell}$ leaves incident with such edges and places them in the set $S_1$. In each of the next $t\ell$ phases, the Presenter partitions $S_i$ into disjoint subsets of size $t\ell+1$ and, for each such subset, reveals a new vertex adjacent to every vertex of that subset. This new vertex has degree $t\ell+1$, and therefore $t+1$ colours must be used on its incident edges. Hence, for each subset, one of these edges must receive a colour from $\{1,2,\dotsc,t\}$. The Presenter then constructs, as in the proof of Theorem \ref{main}, a set $S_{i+1}$ whose vertices are incident only with edges coloured from $\{1,2,\dotsc,t\}$. In the final phase, $S_{t\ell}$ is a set of $t\ell+1$ vertices of degree $t\ell$, which the Presenter joins to a new vertex. Since every vertex in $S_{t\ell}$ has all colours from $\{1,2,\dotsc,t\}$ blocked, while the new vertex has degree $t\ell+1$ and therefore requires $t+1$ colours, colour $2t+1$ must be used, a contradiction.
		\end{proof}
	\end{theorem}

	\section{Conclusions}
	
	Usually, in the context of on-line colouring, the primary metric of interest is the \textit{competitive ratio} of an algorithm, defined as the ratio between the number of colours used by a given algorithm in the on-line setting to the number of colours in an optimal off-line colouring. We consider this ratio in the worst case, i.e. the largest such number when playing against the Presenter's optimal strategy. For majority edge-colourings, however, the performance ratio provides limited insight.
	
	If the parameter $\delta$ is odd, then the competitive ratio for any algorithm using at most five colours is $5/3$. On the one hand, every graph with odd minimum degree requires at least three colours in an off-line majority edge-colouring, since it contains a vertex of odd degree. On the other hand, the constructions in Theorem \ref{main} and Observation \ref{online2} force the Algorithm to use five colours, and the corresponding final graphs can be easily chosen to admit a majority 3-edge-colouring.
	 
	If $\delta$ is even, then the competitive ratio is $5/2$. On the one hand, every graph of minimum degree $\delta\geq 2$ requires at least two colours in a majority edge-colouring. On the other hand, consider the construction in Theorem~\ref{main}, which forces the Algorithm to use five colours. The Presenter follows this construction, and the Algorithm must use five colours since otherwise it loses. We then extend the construction so that all vertices of the resulting graph have even degree by taking two copies of the graph and joining each odd-degree vertex to its counterpart in the other copy. If the resulting graph has an odd number of edges, we can extend it further while preserving the parity of all degrees and the minimum degree $\delta$, by attaching a complete graph $K_{4\ell+3}$ to a single vertex, where $\ell$ is chosen sufficiently large so that the minimum degree remains unchanged. The degrees of all vertices of the resulting graph remain even. Moreover, $K_{4\ell+3}$ has an odd number of edges, and hence the resulting graph has an even number of edges. Therefore, the constructed graph admits a majority 2-edge-colouring by a simple construction which is a consequence of Euler's Theorem (see e.g. Theorem~1 in \cite{Bock}).
	
	Consequently, the primary focus in this setting is to determine the sharp threshold for $\delta$ (as a function of $n$) at which no strategy using only four colours exists for the Algorithm.
	
	\begin{problem}
		Let $\delta\geq 2$. What is the largest $n$ for which the Algorithm has a strategy using at most four colours regardless of the Presenter's strategy, if the final graph is required to have $n$ vertices and minimum degree $\delta$?
	\end{problem}
	
	The construction in Theorem \ref{main} uses $\Theta(\delta^\delta)$ vertices. Determining whether this threshold is optimal in general appears difficult. However, some improvements are possible in the simplest case $\delta=2$. Below, we present some constructions.
	
	\begin{observation}
		For $\delta =2$ and $n\geq 8$, the Presenter has a strategy forcing the Algorithm to use at least five colours.
		\begin{proof}
			Let $n\geq 8$. Suppose that the Algorithm has a strategy allowing it to colour every $n$-vertex graph of minimum degree $2$ using at most four colours.
			
			The Presenter first reveals the edges of a cycle on seven vertices. Since a majority edge-colouring of a subcubic graph is also a proper edge colouring, two distinct colours must be used on the edges incident with every vertex of the cycle. 
			
			Suppose that there are vertices $u$ and $v$ of the cycle $C_7$ whose incident edges use disjoint sets of colours; that is, the edges incident with $u$ receive colours $c_1$ and $c_2$, while those incident with $v$ receive colours $c_3$ and $c_4$, with $c_3,c_4$ distinct from $c_1,c_2$. The Presenter then reveals the edge $uv$. Note that $u$ and $v$ cannot be adjacent on $C_7$, since their sets of colours are disjoint. Colours $c_1$ and $c_2$ are blocked at $u$, while $c_3$ and $c_4$ are blocked at $v$. Since each of the four colours used so far would violate the majority condition at either $u$ or $v$, a fifth colour must be used.
			
			Alternatively, suppose that no pair of vertices of $C_7$ has disjoint colour sets. This implies that for any two vertices, the sets of colours assigned to edges incident with them have at least one common colour. Among the six possible pairs of colours from a set of four colours, at most three such pairs can pairwise intersect. By the Pigeonhole Principle, since there are seven vertices, at least three of them share the same colour set, say $\{c_1,c_2\}$. Hence, colours $c_1$ and $c_2$ are blocked at each of these three vertices. The Presenter now reveals a new, eighth vertex adjacent to each of these three vertices. The three new edges can therefore use only colours $c_3$ and $c_4$. However, to satisfy the majority condition at the new vertex of degree $3$, three distinct colours are required. Hence, a fifth colour must be used, a contradiction.
		\end{proof}
	\end{observation}
	
	Note that for $n\leq 4$, three colours suffice for on-line majority edge-colouring: for $n\leq 3$ it is trivial; for $n=4$, it suffices for the Algorithm to assign the same colour to pairs of non-adjacent edges. However, already for $n=5$ four colours are necessary, since the graph obtained from $K_4$ by subdividing one edge does not admit an off-line majority edge-colouring with only three colours. This leads to the following open problem.
	
	\begin{problem}
		For $n\in\{5,6,7\}$ does the Algorithm have a strategy using at most four colours regardless of the Presenter's strategy, if the final graph is required to have $n$ vertices and minimum degree $\delta =2$? 
	\end{problem}
	
	Note that the greedy algorithm already requires five colours for $n=5$. In particular, the Presenter first reveals the cycle $C_4$, which the greedy algorithm colours using only colours $1$ and $2$. The Presenter then reveals a fifth vertex adjacent to three vertices of the cycle. For each of these three vertices, colours $1$ and $2$ are blocked, while the new vertex requires three distinct colours on its incident edges. Hence, a fifth colour must be used.

\end{document}